\documentclass[11pt, reqno]{amsart}

\usepackage{amsmath, amsthm, amscd, amsfonts, amssymb, graphicx, color, enumerate, xcolor, mathrsfs, latexsym, algorithm, algorithmic, float}
\usepackage{multicol}
\usepackage{enumitem}
\usepackage{pgf,tikz}
\usetikzlibrary{arrows}
\usetikzlibrary{decorations.markings, decorations.pathreplacing}

\usepackage[bookmarksnumbered, colorlinks, plainpages]{hyperref}
\hypersetup{
pdftitle={Research Paper},
pdfauthor={Ali Akbar Yazdan Pour},
pdfsubject={Index of monomial ideals},
pdfcreator={Ali Akbar Yazdan Pour},
colorlinks,
linkcolor=blue,
citecolor=magenta,
anchorcolor=red,
bookmarksopen,
urlcolor=red,
filecolor=red,
}
\theoremstyle{plain}
\newtheorem{theorem}{Theorem}[section]
\newtheorem{theorem*}{Theorem}
\newtheorem{proposition}[theorem]{Proposition}
\newtheorem{lemma}[theorem]{Lemma}
\newtheorem{corollary}[theorem]{Corollary}
\newtheorem{conjecture}[theorem]{Conjecture}

\theoremstyle{definition}

\newtheorem{question}{Question}

\newtheorem{remark}{Remark}

\theoremstyle{remark}

\newcommand{\B}{\mathcal{B}}
\newcommand{\C}{\mathcal{C}}
\newcommand{\D}{\mathcal{D}}
\newcommand{\G}{\mathcal{G}}
\renewcommand{\O}{\mathcal{O}}
\renewcommand{\S}{\mathcal{S}}
\newcommand{\T}{\mathcal{T}}

\newcommand{\KK}{\mathbb{K}}

\newcommand{\lcm}{\mathrm{lcm}}
\newcommand{\gen}[1]{\langle#1\rangle}
\newcommand{\set}[1]{\left\lbrace#1\right\rbrace}
\newcommand{\ind}{\mathrm{index}}
\newcommand{\supp}{\mathrm{supp}}
\newcommand{\bx}{\textbf{x}}

\tikzset{->-/.style={decoration={markings, mark=at position 0.5 with {\arrow{stealth}}},postaction={decorate}}}
\tikzset{-<-/.style={decoration={markings, mark=at position 0.5 with {\arrow{stealth reversed}}},postaction={decorate}}}
\begin{document}
\title[Index of $3$-uniform clutters]{Green-Lazarsfeld index of square-free monomial ideals and their powers}

\author[M. Farrokhi D. G.]{Mohammad Farrokhi D. G.}
\email{m.farrokhi.d.g@gmail.com,\ farrokhi@iasbs.ac.ir}
\address{ Research Center for Basic Sciences and Modern Technologies (RBST), Institute for Advanced Studies in Basic Sciences (IASBS), 
Zanjan 45137-66731, Iran}

\author[Y. Sadegh]{Yasin Sadegh}
\email{yassin.sadegh@yahoo.com, y.sadegh@uma.ac.ir}
\address{Department of Mathematics,  Faculty of Mathematical
Sciences, University of Mohaghegh Ardabili, 56199-11367,  Ardabil, Iran.}  

\author[A. A. Yazdan Pour]{{Ali Akbar} {Yazdan Pour}}
\email{yazdan@iasbs.ac.ir}
\address{Department of Mathematics, Institute for Advanced Studies in Basic Sciences (IASBS), Zanjan 45137-66731, Iran}

\begin{abstract}
Let $\mathbb{K}$ be a field and $I$ be a square-free monomial ideal in the polynomial ring $\mathbb{K}[x_1, \ldots, x_n]$. The Green-Lazarsfeld index, $\mathrm{index}(I)$, counts the number of steps to reach to a syzygy minimally generated by a nonlinear form in a graded minimal free resolution of $I$. In this paper, we study this invariant for $I$ and its powers from a combinatorial point of view. We characterize all square-free monomial ideals $I$ generated in degree $3$ such that $\mathrm{index}(I)>1$. Utilizing this result, we also characterize all square-free monomial ideals generated in degree $3$ such that $\mathrm{index}(I)>1$ and $\mathrm{index}(I^2)=1$. In case $n\leq5$, it is shown that $\mathrm{index}(I^k)>1$ for all $k$ if $I$ is any square-free monomial ideal with $\mathrm{index}(I)>1$.
\end{abstract}
\maketitle
\section*{Introduction}
Let $\KK$ be a field and $S=\mathbb{K}[x_1,\ldots,x_n]$ be the polynomial ring in $n$ variables endowed with the standard grading (i.e. $\deg(x_i)=1$ for all $1\leq i\leq n$). Let $I \neq 0$ be a homogeneous ideal in $S$ generated by elements of degree $d$. A classical problem in algebraic geometry and commutative algebra is to study equations of syzygies of $I$ and determine when these equations are linear forms. The ideal $I$ is said to satisfy the \textit{${\rm N}_{d,p}$-property} ($p>0$), if
\[ \beta_{i, i+j} (I) =0, \quad \text{for all } i < p, \text{ and } j>d.\]
The quantity
\[\mathrm{index}(I) = \sup \left\{p \colon \quad I \text{ satisfies $\mathrm{N}_{d,p}$-property} \right\}\]
is called the \textit{Green-Lazarsfeld index}, or simply \textit{index} of $I$. The ${\rm N}_{d,p}$ notation defined in \cite{EGHP} comes essentially from the notation $N_p$ of Green and Lazarsfeld in \cite{mlg-rl-84, mlg-rl-85} (see also \cite{EL93}). This notation indicates when the minimal equations defining the syzygies are in the simplest form (that is of the linear form). The Green-Lazarsfeld index of ideals is very difficult to compute in general. Important conjectures, such as Green's conjecture \cite[Chapter 9]{Eisenbud 2005}, predicts the value of this invariant for certain families of varieties. In \cite{wb-ac-tr} the authors study the Green-Lazarsfeld index of the Veronese embeddings $\nu_c\colon \mathbb{P}^{n-1} \to \mathbb{P}^N$ of degree $c$ of projective spaces and, more generally, of the Veronese embeddings of arbitrary varieties (see also \cite{mg, tj-pp-jw, go-rp, er, jw} in this matter). In their interesting paper \cite{EGHP}, Eisenbud, Green, Hulek, and Popescu provide some nice examples and conjectures about $\mathrm{N}_{2,p}$-property. As a remarkable result, they characterize the property $\mathrm{N}_{2,p}$ for monomial ideals in degree $2$ \cite[Theorem 2.1, Proposition 2.3]{EGHP}. It turns out that the edge ideal $I(G)$ of a graph $G$ satisfies $N_{2,p}$-property if and only if every induced cycle in $\bar{G}$ has length $\geq p+3$. In particular, 
\[\mathrm{index}(I(G)) = \inf \left\{ |C|-3 \colon \quad C \text{ is an induced cycle in } \bar{G} \text{ of length}>3 \right\}.\]
Though we have such a good combinatorial formulation for the index of square-free monomials generated in degree $2$, the precise value of index of powers of such ideals is still mysterious. In this regard, we have the following result due to Bigdeli, Herzog, and Zaare-Nahandi \cite[Theorem 2.1]{mb-jh-rzn}.
\begin{theorem*} \label{main theorem of MJR}
Let $I=I(G)$ be the edge ideal of a graph $G$, where $G$ may have some loops. The following conditions are equivalent:
\begin{itemize}
\item[\rm (a)] $G$ is gap-free, i.e. no induced subgraph of $G$ consists of two disjoint edges;
\item[\rm (b)] $\mathrm{index}(I^k)>1$, for all $k$;
\item[\rm (c)] $\mathrm{index}(I^k)>1$, for some $k$.
\end{itemize} 
\end{theorem*}
The above theorem implies in particular that for a monomial ideal generated in degree $2$ we have
$\mathrm{index}(I) = 1$ if and only if $\mathrm{index}(I^k) = 1$ for all $k$, while simple examples show that this fails if $I$ is not generated in degree $2$ (see Theorem \ref{index(I)>1 and index(I2)>1}).

In this paper, we consider the square-free monomial ideals $I$ generated in degree $3$ and study the index of $I$ and its powers. Note that such ideals are in one-to-one correspondence with edge ideals of $3$-uniform clutters. Recall that a $3$-uniform clutter $\C$ on vertex set $V=\{v_1, \ldots, v_n\}$ is a collection of $3$-subsets of $V$ such that $V=\cup\C$, and the edge ideal $I(\C)$ of $\C$ is the ideal generated by all monomials $x_ix_jx_k$ such that $\{v_i,v_j,v_k\} \in \C$. In view of the above theorem, it is reasonable to ask whether for the edge ideal $I=I(\C)$ of a $3$-uniform clutter $\C$, the property $\mathrm{index}(I)>1$ is translated to a combinatorial property of $\C$. To this end, we interpret the above theorem as follows: If $I=I(G)$ is the edge ideal of a graph $G$,
then the followings are equivalents:
\begin{itemize}
\item[(i)] the complement $\bar{G}$ of $G$ is $C_4$-free, where $C_4$ stands for the cycle of length $4$;
\item[(ii)] $\mathrm{index}(I)>1$;
\item[(iii)] $\mathrm{index}(I^k)>1$, for some $k$.
\end{itemize}
Here we consider $C_4$ as the minimal triangulation of $1$-sphere that is not the triangle. Note that the equivalence of (i) and (ii) has been proved first in \cite[Corollary 2.9]{hd-ch-js}. We show in Theorem~\ref{index of gap-free} and Proposition~\ref{index of some power} that if $I=I(\C)$ is the edge ideal of a $3$-uniform clutter $\C$, then the followings are equivalent:
\begin{itemize}
\item[(i)] the complement $\bar{\C}$ of $\C$ is $\mathscr{C}$-free;
\item[(ii)] $\mathrm{index}(I)>1$;
\item[(iii)] $\mathrm{index}(I^k)>1$, for some $k$,
\end{itemize}
where $\mathscr{C}$ refers to a set of four $3$-uniform clutters whose three of them come from the minimal triangulation of $2$-sphere that is not the tetrahedron (see Remark~\ref{betti numbers of bi-pyramids} for precise definition of $\mathscr{C}$). Indeed, as we will see later, the condition that $\bar{\C}$ is $\mathscr{C}$-free has strong implications for the lcm-lattice. As a byproduct of the above result, it follows that $\ind(I(\C))>1$ if $\C$ is a $3$-uniform clutter with $|\C|>\binom{|V(\C)|}{3}-6$. We will see in Proposition~\ref{mu(I)>C(n,d)-2d} that indeed $\ind(I(\C))=\infty$ if $\C$ is a $d$-uniform clutter with $|\C|>\binom{|V(\C)|}{d}-2d$.

If $I$ is a (square-free) monomial ideal generated in degree $2$ with $\mathrm{index}(I)>1$, then Theorem~\ref{main theorem of MJR} implies that $\mathrm{index}(I^k)>1$ for all $k$. Examples show that the same statement does not hold for ideals generated in degree greater than $2$. However, in Theorem~\ref{index with n <= 5} we show that for any square-free monomial ideal $I\subseteq\KK[x_1,\ldots,x_5]$ with $\mathrm{index}(I)>1$ we have $\mathrm{index}(I^k)>1$ for all $k$. Moreover, in Theorem~\ref{index(I)>1 and index(I2)>1} we characterize all square-free monomial ideals generated in degree $3$ such that $\mathrm{index}(I)>1$ and $\mathrm{index}(I^2)=1$. It turns out that if $I=I(\C)$ is the edge ideal of a $3$-uniform clutter $\C$ and $\mathrm{index}(I)>1$, then $\mathrm{index}(I^2)>1$ if and only if $\C$ is $\mathscr{D}$-free where $\mathscr{D}$ is the family of clutters defined by Figures \ref{d=3, n=6}, \ref{d=3, n=8}, \ref{d=3, n=7}.
\section{Preliminaries}
Throughout, let $\mathbb{K}$ be a field and $S=\mathbb{K}[x_1,\ldots,x_n]$ be the polynomial ring over $n$ variables. In this section, we quickly review basic notions and preliminaries that we will meet in the sequel.
\subsection{Multigraded Betti numbers}
Let $I \subset S$ be a monomial ideal in the polynomial ring $S$. Then $I$ is a multigraded $S$-module and so it admits a minimal multigraded free $S$-resolution
$$\cdots \to F_2 \to F_1 \to F_0 \to I \to 0,$$
where $F_i = \bigoplus_{\textbf{\rm \textbf{a}} \in \mathbb{Z}^n} S \left( -\textbf{\rm \textbf{a}} \right)^{\beta_{i,\textbf{\rm \textbf{a}}}(I)}$. The numbers $\beta_{i,\textbf{\rm \textbf{a}}}(I) = \dim_\mathbb{K} \mathrm{Tor}^S_i \left( \mathbb{K}, I \right)_{\textbf{\rm \textbf{a}}}$ are called the \textit{multigraded Betti numbers} of $I$.
For $\textbf{\rm \textbf{a}}=(a_1, \ldots, a_n) \in \mathbb{Z}^n$, let $|\textbf{\rm \textbf{a}}|= \sum_i a_i$. Define
\[\beta_{i,j}(I) = \sum\limits_{\textbf{\rm \textbf{a}} \in \mathbb{Z}^n \atop |\textbf{\rm \textbf{a}}|=j} \beta_{i, \textbf{\rm \textbf{a}}} (I).\]
The numbers $\beta_{i,j} (I)$ are the \textit{graded Betti numbers} of $I$ with respect to the standard grading on $S$ (i.e. $\deg(x_i)=1$). A monomial ideal $I \subset S$ is said to have a \textit{$d$-linear resolution} if $\beta_{i,j} (I) =0$, for all $i,j$ with $j-i\neq d$. 

\subsection{Green-Lazarsfeld index}
The Green-Lazarsfeld index (or simply index) counts the number of linear steps in the graded minimal free resolution of an ideal. The monomial ideal $I$ is \textit{$r$-step linear} if $I$ has a linear resolution up to homological degree $r$; in other words, if $I$ is generated by homogeneous elements in degree $d$ and $\beta_{i,j}(I)=0$ for all pairs $(i, j)$ with $0 \leq i \leq r$ and $j-i > d$. The quantity
\[ \ind(I) = \sup\{ r \colon\quad I \text{ is $r$-step linear} \}+ 1\]
is called the \textit{index} of $I$. In particular, $I$ has a linear resolution if and only if $\ind(I) = \infty$. A monomial ideal $I$ is called \textit{linearly presented} if $\ind(I)>1$.

\subsection{Clutters and their associated ideals}
A \textit{clutter} $\mathcal{C}$ on vertex set $V=\{v_1, \ldots, v_n\}$ is a collection of subsets of $V$, called \textit{circuits} of $\mathcal{C}$, such that $V=\cup\C$ and if $F_1$ and $F_2$ are distinct circuits, then $F_1 \nsubseteq F_2$.
A clutter $\C$ is called \textit{$d$-uniform} if every circuit of $\C$ has $d$ vertices.  Let $\textbf{x}_\varnothing = 0$ and $\textbf{x}_F=\prod_{v_i \in F}{x_i}$ for any non-empty subset $F$ of $V$. For a non-empty clutter $\mathcal{C}$ we define the ideal $I \left( \mathcal{C} \right)$ to be
\[I(\mathcal{C}) = \left(  \textbf{x}_T \colon \quad T \in \mathcal{C} \right),\]
and we  set $I(\varnothing) = 0$. The ideal $I \left( {\mathcal{C}} \right)$ is called the \textit{edge ideal} of $\mathcal{C}$. 
%
%
If $\mathcal{C}$ is a $d$-uniform clutter on $V$, then the \textit{complement} $\bar{\mathcal{C}}$ of $\mathcal{C}$ is defined to be
\begin{equation*}
\bar{\mathcal{C}} = \{F \subset V \colon \quad |F|=d, \,F \notin \mathcal{C}\}.
\end{equation*}
Let $\mathcal{C}$ be a clutter and let $e$ be a subset of $V$. By $\mathcal{C} \setminus e$ we mean the clutter $\left\{F \in \mathcal{C} \colon \ e \nsubseteq F \right\}$.


\subsection{Simplicial complexes and Stanley-Reisner ideals}
A \textit{simplicial complex} $\Delta$ on the vertex set $V=\{v_1, \ldots, v_n\}$ is a collection of subsets of $V$, called \textit{faces}, such that $\{ v_i \} \in \Delta$  for all $i$ and, $F \in \Delta$ implies that all subsets of $F$ are also in $\Delta$. Maximal faces of $\Delta$ are called \textit{facets} of $\Delta$. A simplicial complex with facets $F_1,\ldots,F_m$ is denoted by $\gen{F_1,\ldots,F_m}$. The \textit{induced} subcomplex $\Delta[W]$ of $\Delta$ is the simplicial complex $\Delta[W]=\{F \in \Delta \colon \; F \subseteq W \}$. Also, the Stanley-Reisner ideal $I_\Delta$ is the square-free monomial ideal generated by all $\bx_F$ with $F\notin\Delta$.

Let $\mathcal{C}$ be a $d$-uniform clutter on the vertex set $[n]:=\{1,\ldots,n\}$. A subset $F$ of $[n]$ is a \textit{clique} in $\C$ if either $|F|<d$ or else all $d$-subsets of $F$ belong to $\C$. The \textit{clique complex} $\Delta (\C )$ of $\mathcal{C}$ is the simplicial complex whose faces are cliques of $\mathcal{C}$. It is easy to see that $I(\C)=I_{\Delta(\bar{\C})}$.

\subsection{Gasharov-Peeva-Welker formula}
In the following, by a \textit{poset} we mean a partially ordered set. A subset $C$ of a poset $P$ is called a \textit{chain} in $P$ if $C$ is a totally ordered subset of $P$ with respect to the induced order.

Let $I$ be a monomial ideal and $\mathcal{G}(I) = \{u_1, \ldots, u_m\}$ be the unique minimal
set of monomial generators of $I$. We denote by $L(I)$ the \textit{lcm-lattice} of $I$, that is a poset whose elements are labeled by the least common multiples of subsets of monomials in $\mathcal{G}(I)$ ordered by divisibility. For any $u \in L(I)$ we denote by $(1, u)$ the open interval of $L(I)$, which is by definition
\[(1,u)=\{ v \in L(I) \colon \; 1<v<u \}. \]
Furthermore, we denote by $\Delta_I(1, u)$ the order complex of the poset $(1, u)$, that is the simplicial complex whose faces are the chains in the poset $(1, u)$.
\begin{theorem}[{\cite[Theorem 2.1]{GPW}}]\label{Gasharov-Peeva-Welker formula}
For all $i \geq 0$ and all $u \in L(I)$,
\[ \beta_{i,u}(I) = \dim_\KK \tilde{H}_{i-1} \left( \Delta_I(1, u);\KK \right).\]
\end{theorem}

\section{Index of equigenerated square-free monomial ideals and their powers}
Let $\C$ be a clutter on vertex set $V$ and $W$ be a subset of $V$. The \textit{induced} subclutter $\C[W]$ of $\C$ is the clutter $\C[W]=\{F \in \C \colon \; F \subseteq W \}$. Assume that $\mathscr{F}$ is a family of clutters. We say that $\C$ is \textit{$\mathscr{F}$-free} if $\C$ does not contain an induced subclutter isomorphic to any element in  $\mathscr{F}$. Let $G$ be a graph and $I=I(G)$ be the edge ideal of $G$. It follows from \cite[Corollary 2.9]{hd-ch-js} that $\bar{G}$ is $C_4$-free if and only if $\mathrm{index}(I)>1$. We intend to present a similar result for $3$-uniform clutters (Theorem \ref{index of gap-free}). To do this, we need the following setting: Let $I$ be a monomial ideal whose minimal generators are all of degree $d$. In \cite{mb-jh-rzn} the authors define a graph $G_I$ whose vertex set is $\mathcal{G}(I)$ and for which $\{u, v\}$ is an edge of $G_I$ if and only if $\deg(\lcm(u, v)) = d + 1$. For all $u, v \in \mathcal{G}(I)$, let $G_I^{u,v}$ be the induced subgraph of $G_I$ with vertex set
\[\{w \in \mathcal{G}(I) \colon \; w \text{ divides } \lcm(u, v)\}.\]

The following result obtained simply from the proofs of \cite[Proposition 1.1 and Corollary 1.2]{mb-jh-rzn} and also Theorem \ref{Gasharov-Peeva-Welker formula} plays a crucial role in this paper. 

\begin{theorem} \label{index in terms of graph}
Let $I$ be a monomial ideal generated in degree $d$. Then 
\begin{itemize}
\item[\rm (i)]if $u,v\in\G(I)$ are such that $\deg\lcm(u,v)>d+1$ and $\Delta_I(1,\lcm(u,v))$ is disconnected, then $u,v$ belong to distinct connected components of $G_I^{u,v}$;
\item[\rm (ii)]if $\Delta_I(1,\lcm(u,v))$ is connected for all $u,v\in\G(I)$, then so is $G_I^{u,v}$ for all $u,v\in\G(I)$.
\end{itemize}
In particular, $\ind(I)>1$ if and only if for all $u, v \in \mathcal{G}(I)$ there is a path in $G_I^{u,v}$ connecting $u$ and $v$.
\end{theorem}

\begin{remark} \label{betti numbers of bi-pyramids}
Let $\B$ be the following bi-pyramid, $\B_1=\B \cup \{125\}$, and $\B_2=\B \cup \{125, 135\}$. Since $\gen{\B}$ is a triangulation of the $2$-sphere, $\beta_{1,5}(I(\bar{\B}))=\dim_\KK\tilde{H}_2(\Delta(\bar{\B});\KK)=1$. Hence, by \cite[Theorem 2.1]{BYZ}, we have
\[1=\beta_{1,5}(I(\bar{\B}))=\beta_{1,5}(I(\bar{\B}_1))=\beta_{1,5}(I(\bar{\B}_2)).\]

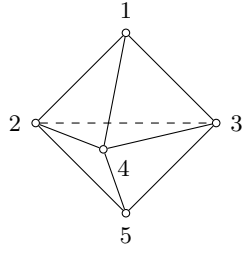
\begin{center}
\begin{figure}[!htp]
\begin{tikzpicture}[scale=1]
\node [draw, circle, fill=white, inner sep=1pt, label=above:\footnotesize{$1$}] (1) at (5.0,3.2) {};
\node [draw, circle, fill=white, inner sep=1pt, label=left:\footnotesize{$2$}] (2) at (3.8,2.0) {};
\node [draw, circle, fill=white, inner sep=1pt, label=right:\footnotesize{$3$}] (3) at (6.2,2.0) {};
\node [draw, circle, fill=white, inner sep=1pt, label=340:\footnotesize{$4$}] (4) at (4.7,1.65) {};
\node [draw, circle, fill=white, inner sep=1pt, label=below:\footnotesize{$5$}] (5) at (5.0,0.8) {};
\draw (1)--(2)--(4)--(3)--(1)--(4)--(5)--(3) (2)--(5);
\draw [dashed] (2)--(3);
\end{tikzpicture}
\end{figure}
\end{center}
On the other hand, if $\B'=\binom{[6]}{3}\setminus\{123, 456\}$ then one may easily check that $\beta_{1,6}(I(\bar{\B}'))=1$. Let $\mathscr{C}$ be the family $\{ \B,\ \B_1,\ \B_2,\ \B'\}$ of clutters. The following theorem shows that this family of clutters tends to a characterization of all square-free monomial ideals $I$ such that $\ind (I) >1$.
\end{remark}

\begin{theorem}\label{index of gap-free}
Let $\C$ be a $3$-uniform clutter and let $I=I(\C)$ be its edge ideal. Let $\mathscr{C}$ be the family of clutters as in the above remark. The following conditions are equivalent:
\begin{itemize}
\item[\rm (a)] $\bar{\C}$ is $\mathscr{C}$-free;
\item[\rm (b)] $\ind (I) >1$, that is $I$ is linearly presented;
\item[\rm (c)] $\beta_{1,5} (I)= \beta _{1,6} (I)=0$.
\end{itemize}
\end{theorem}

\begin{proof}
$a \Rightarrow b$: Suppose that $\bar{\C}$ is $\mathscr{C}$-free. Using Theorem~\ref{index in terms of graph}, it is enough to show that for all $u,v \in \mathcal{G}(I)$ there is a path in $G:=G_{I}^{u,v}$ between $u$ and $v$. Let $u= x_i x_j x_k$ and $v=x_r x_s x_t$. 

If $\deg\lcm(u,v)=4$, then $\{u,v\} \in E(G)$ and we have nothing to prove. Suppose that $\deg \lcm(u,v)=5$. Without loss of generality, we may assume that $u=x_1 x_2 x_3$ and $v=x_1 x_4 x_5$. Suppose on the contrary that there is no path in $G$ between $u$ and $v$. Let $X= \{124, 125, 134, 135\}$. If $\bx_F$ is a vertex of $G$ for some $F \in X$, then $u-\bx_F-v$ is a path in $G$ between $u$ and $v$, a contradiction. Thus $X \subseteq \bar{\C}$. Now put
\begin{align*}
A_1&= \set{x_2 x_3 x_4, x_2 x_4 x_5},&&&A_2&= \set{x_2 x_3 x_4, x_3 x_4 x_5},\\
A_3&= \set{x_2 x_3 x_5, x_2 x_4 x_5},&&&A_4&= \set{x_2 x_3 x_5, x_3 x_4 x_5}.
\end{align*}
Observe that $A_i\nsubseteq V(G)$, for $i=1,\ldots,4$ for otherwise $A_i\cup\{u,v\}$ yields a path between $u$ and $v$ in $G$. Hence $\bar{\C}\cap A_i\neq\varnothing$, for $i=1,\ldots,4$. A simple verification shows that the clutter induced by $X$ and those circuits $F\in\bar{\C}$ with $\bx_F\in A_1\cup\cdots\cup A_4$ contains one of the elements of $\mathscr{C}$ contradicting the $\mathscr{C}$-freeness of $\bar{\C}$.

Next assume that $\deg \lcm(u,v)=6$. Without loss of generality, assume that $u=x_1 x_2 x_3$ and $v=x_4 x_5 x_6$. We claim that there is a path in $G$ connecting
$u$ to $v$ (of length $3$). Suppose on the contrary that there is no path between $u$ and $v$ in $G$. Let $B_{ijab}=\{x_i x_j x_a, x_i x_a x_b\}$ for all $1\leq i<j\leq 3$ and $4\leq a<b\leq 6$. Then $B_{ijab}\nsubseteq V(G)$ because otherwise $B_{ijab}\cup\{u,v\}$ is a path of length $3$ in $G$ between $u$ and $v$, contradicting the assumption. The group $H=\gen{(1\ 2\ 3), (1\ 2), (4\ 5\ 6), (4\ 5)}$ acts naturally on the set
\[\Omega:=\set{x_i x_j x_a\colon\ 1\leq i<j\leq3,\ 4\leq a\leq 6}\]
by $(x_ix_jx_a)^\pi=x_{i^\pi}x_{j^\pi}x_{a^\pi}$ for all $\pi\in H$ and $x_ix_jx_a\in\Omega$. Suppose $X:=\Omega\cap V(G)$ is a $k$-subset of $\Omega$, for some $0\leq k\leq9$. Then $X$ belongs to an orbit of $H$ on $\binom{\Omega}{k}$, the set of all $k$-subsets of $\Omega$. Since all elements of an orbit share the same properties, we can restrict ourselves to a fixed representative of any orbit. Observe that $G$ has $1, 1, 3, 6, 7, 7, 6, 3, 1, 1$ orbits in action on $\binom{\Omega}{k}$, for $k=0,\ldots,9$, respectively, so that we have $36$ cases to consider. For any such case, we have $iab\in\bar{\C}$ (resp. $ija\in\bar{\C}$) whenever $x_ix_jx_a\in X$ (resp. $x_ix_jx_a\in\Omega\setminus X$) because $B_{ijab}\nsubseteq V(G)$. What remains is to consider all possibilities $iab\in\C$ or $iab\in\bar{\C}$ for all $iab$ satisfying $x_ix_jx_a\in X$ whose the number of them is bounded above by $2^5=32$ except for the case where $X=\varnothing$. A simple computer program shows that we have totally $105$ cases and, in any case, $\bar{\C}$ has an element of $\mathscr{C}$ as an induced subclutter. This contradicts the fact that $\bar{\C}$ is $\mathscr{C}$-free.

$(b) \Rightarrow (c)$: This implication is clear by the definition of the index of an ideal.

$(c) \Rightarrow (a)$: Suppose on the contrary that $\bar{\C}$ is not $\mathscr{C}$-free. So $\bar{\C}$ contains an induced subclutter $\D$ isomorphic to one of the elements of $\mathscr{C}$. If $\D \ncong{[6] \choose 3} \setminus \{123, 456\}$, then in view of Remark~\ref{betti numbers of bi-pyramids}, we obtain 
$$\beta_{1,5}(I) \geq \beta_{1,5}(I(\bar{\D}))=1,$$
which contradicts to our assumption. Thus $\D\cong{[6] \choose 3} \setminus \{123, 456\}$. Without loss of generality, assume that $\D={[6] \choose 3} \setminus \{123, 456\}$. Then by setting $\alpha=x_1 \cdots x_6$, the interval $(1,\alpha)$ consists of only two points $x_1x_2x_3$ and $x_4x_5x_6$. Thus $\Delta(1,\alpha)$ is disconnected. It follows that $\beta_{1,6} (I) \geq \beta_{1, \alpha} (I) = \dim \tilde{H}_0(\Delta(1, \alpha); \mathbb{K})=1$, which is a contradiction.
\end{proof}

Theorem \ref{index of gap-free} implies that $\ind(I(\C))>1$ for any $3$-uniform clutter $\C$ on a vertex set $[n]$ with $|\C|>\binom{n}{3}-6$. In the following, we show that an stronger result holds under the same assumptions. This result is not restricted only to $3$-uniform clutters.
\begin{proposition}\label{mu(I)>C(n,d)-2d}
Let $d>1$ be a positive integer and $\kappa_d$ be the smallest size of a $d$-uniform clutter $\C$ such that $I(\bar{\C})$ does not have linear resolution. Then $\kappa_d=2d$. In particular, if $\C$ is a $d$-uniform clutter with $|\C| > \binom{|V(\C)|}{d}-2d$, then $I(\C)$ has a linear resolution. 
\end{proposition}
\begin{proof}
First we show that $\kappa_d\leq2d$. Let
\[\C_d=\left(\binom{\{1,\ldots,d+1\}}{d}\cup\binom{\{2,\ldots,d+2\}}{d}\right)\setminus\{2,\ldots,d+1\}.\]
The simplicial complex $\gen{\C_d}$ is a triangulation of the $(d-1)$-sphere $\mathbb{S}^{d-1}$ so that
\[\tilde{H}_{d-1}(\Delta(\C_d);\KK)=\tilde{H}_{d-1}(\gen{\C_d};\KK)=H_{d-1}(\mathbb{S}^{d-1};\KK)\neq0.\]
Thus $\beta_{1,d+2}(I(\bar{\C}_d))\neq0$ by Hochster's formula \cite[Theorem 5.1]{mh}, and hence $I(\bar{\C}_d)$ does not have linear resolution showing that $\kappa_d\leq2d$. 

To complete the proof, let $\C$ be a $d$-uniform clutter such that $|\C|<2d$. Let $\Delta=\Delta(\C)$ be the clique complex of $\C$. It is obvious that $d-1\leq\dim\Delta\leq d$. 
Suppose $\partial_{d-1}(\sum_{i=1}^mc_iF_i)=0$ for some circuits $F_1,\ldots,F_m$ of $\C$ and put $\C'=\{F_1,\ldots,F_m\}$. Clearly, every $(d-1)$-subset of $F_i$ is contained in $F_j$ for some $j\neq i$. Let $X=\cup\C'$. We show that any $t$-subset of $X$ is contained in some circuit of $\C'$ provided that $t\leq d$. Clearly, the result holds for $t=1$. Assume the claim is true for some $t<d$ but not for $t$. Let $X'$ be a $t$-subset of $X$. By assumptions, every $(t-1)$-subset of $X'$ is contained in at least $d-t+2$ circuits $F_{k_1}^{X'},\ldots, F_{k_{d-t+2}}^{X'}$ of $\C'$ for $X'$ is contained in some circuit $F$ and $F$ has $d-t+2$ subsets of size $d-1$ containing $X'$. Since $F_{k_i}^{X'}$ and $F_{k_j}^{X''}$ are distinct for $X'\neq X''$ or $i\neq j$, it follows that 
\[t(d-t+2)=\#\set{F_{k_i,X'}\colon\ X'\in\binom{X}{t-1},\ 1\leq i\leq d-t+2}\leq|\C'|<2d.\]
This yields $d<t$, which is a contradiction. It follows that $\C'=\binom{X}{d}$, hence $X\in\Delta$ and $|X|=d+1$. This shows that $\sum_{i=1}^mc_iF_i$ is a scalar multiple of $\partial_d(X)$ and consequently $\tilde{H}_{d-1}(\Delta[W],\KK)=0$ for all $W\subseteq V(\C)$.

If $\dim\Delta=d-1$, then $\tilde{H}_i(\Delta[W];\KK)=0$ by the argument above for all $W\subseteq V(\C)$ and $i\geq d-1$. Hence $I(\bar{\C})=I_{\Delta(\C)}$ has a linear resolution (see e.g. \cite[Proposition 3.1]{yazdan}). So, assume that $\dim\Delta=d$. In this case, $\C$ has a unique clique $F$ of size $d+1$. Indeed, if $F$ and $F'$ are two distinct cliques in $\C$ of size $d+1$, then $|F\cap F'|\leq d$ so that $F\cup F'$ contains at least $2d+1$ circuits of $\C$ contradicting the fact that $|\C|<2d$. Since $\dim\Delta=d$, it follows that $\tilde{H}_d(\Delta;\KK)=\ker\partial_d=0$, where $\partial_d$ is the chain map from $\KK F$ into $\bigoplus_{e\in\C}\KK e$. Consequently, $\tilde{H}_i(\Delta[W];\KK)=0$ for all $W\subseteq V(\C)$ and $i\geq d-1$. Thus $I(\bar{\C})$ has a linear resolution in this case too.
%
\end{proof}

If $G$ is a graph and $\ind(I(G)^k)>1$ for some $k \geq 1$, then $\bar{G}$ is $C_4$-free (Theorem~\ref{main theorem of MJR}). Here we prove an analogue result for $3$-uniform clutters.
\begin{proposition}\label{index of some power}
Let $\C$ be a $3$-uniform clutter and $I=I(\C)$ be the corresponding edge ideal. The following conditions are equivalent:
\begin{itemize}
\item[\rm (a)]$\bar{\C}$ is $\mathscr{C}$-free;
\item[\rm (b)]$\ind(I^k)>1$, for some $k \geq 1$.
\end{itemize}
\end{proposition}
\begin{proof}
(a) $\Rightarrow$ (b). If $\bar{\C}$ is $\mathscr{C}$-free, then by Theorem \ref{index of gap-free}, $\mathrm{index}(I)>1$.

(b) $\Rightarrow$ (a). If $k=1$ the result holds in view of Theorem~\ref{index of gap-free}, so assume that $k>1$ and $\ind(I^k) >1$. Suppose on the contrary that $\bar{\C}$ has an induced subclutter isomorphic to an element of $\mathscr{C}$. First assume that $\bar{\C}$ contains the bi-pyramid 
\[\B=\{123, 124, 134, 235, 345, 245\}\]
as induced subclutter. It follows that $x_1 x_2 x_5 , x_1 x_3 x_5, x_1 x_4 x_5, x_2 x_3 x_4 \in I$. Note that the monomials $(x_2 x_3 x_4)^k, (x_1 x_2 x_5)(x_2 x_3 x_4)^{k-1}, (x_1 x_3 x_5)(x_2 x_3 x_4)^{k-1}$, and $(x_1 x_4 x_5)(x_2 x_3 x_4)^{k-1}$ are atoms in the lcm-lattice of $I^k$.  Clearly $\alpha :=x_1 {x_2}^k {x_3}^k {x_4} ^k x_5$ covers these four atoms and the open interval $(1, \alpha)$ consists of the four atoms. Hence $\Delta(1, \alpha)$ is disconnected. Consequently,
\[\beta_{1,3k+2} (I^k) \geq \beta_{1, \alpha} (I^k) = \dim_{\KK} \left( \tilde{H}_0( \Delta(1, \alpha);\KK )\right)\neq 0.\]
This implies that $\ind(I^k)=1$, a contradiction. If $\B_1=\B \cup \{125\}$ or $\B_2=\B \cup \{125, 135\}$ is an induced subclutter of $\bar{\C}$, then the same discussion as above leads us to a contradiction. Finally, suppose that $\bar{\C}$ contains $\binom{[6]}{3} \setminus \{123,456\}$ as induced subclutter. In this case, put $\alpha=(x_1 x_2 x_3)^k (x_4 x_5 x_6)$. Then the open interval $(1, \alpha)$ consists of two atoms $(x_1 x_2 x_3)^{k-1}(x_4 x_5 x_6)$ and $(x_1 x_2 x_3)^k$. Thus $\Delta(1, \alpha)$ is disconnected and
\[\beta_{1,3k+3} (I^k) \geq \beta_{1, \alpha} (I^k) = \dim_{\KK} \left( \tilde{H}_0( \Delta(1, \alpha);\KK) \right)\neq 0.\]
This implies that $\ind(I^k)=1$, which is again a contradiction.
\end{proof}

If $I=I(G)$ is the edge ideal of a graph $G$ and $\ind(I)>1$, then Theorem~\ref{main theorem of MJR} implies that $\ind(I^k)>1$ for all $k\geq1$. Examples show that this statement is not true if $I$ is generated by (square-free) monomials generated in degree greater that $2$ (see Theorem~\ref{index(I)>1 and index(I2)>1}). The next theorem shows that this result still holds for square-free monomial ideals in a polynomial ring with at most five variables. Theorem~\ref{index(I)>1 and index(I2)>1} shows that we can not extend this result for polynomial rings in more than five variables.
\begin{theorem}\label{index with n <= 5}
Let $I\subseteq \KK[x_1,\ldots,x_n]$ be a square-free monomial ideal with $\ind(I)>1$, where $n\leq5$. Then $\ind(I^k)>1$ for all $k\geq1$.
\end{theorem}
\begin{proof}
Suppose $I\subseteq \KK[x_1,\ldots,x_n]$ is a square-free monomial ideal generated in degree $d$ and that $x_1\cdots x_n\mid\lcm(\G(I))$. If $d=1$, then there is nothing to prove. The case $d=2$ follows from Theorem~\ref{main theorem of MJR}. If $d=n-1$ or $d=n$, then $I$ is a polymatroidal ideal and hence all powers of $I$ have linear resolutions (see \cite[Theorems 5.2 and 5.3]{ac-jh}). So, the only case we need to consider is $n=5$ and $d=3$. Suppose on the contrary that $\ind(I^m)=1$ for some $m$, and that $m$ is minimum with this property. Then $\ind(I^i)>1$, for all $i=1,\ldots,m-1$. By Theorem~\ref{index in terms of graph}, there exist $u,v\in\G(I^m)$ such that $u$ and $v$ belong to different connected components of $G_{I^m}^{u,v}$. Let $u=u_1\cdots u_m$ and $v=v_1\cdots v_m$ with $u_i,v_i\in\G(I)$, for $i=1,\ldots,m$. Let $U=\{u_1,\ldots,u_m\}$ and $V=\{v_1,\ldots,v_m\}$. It turns out that $U\cap V=\varnothing$ for $G_{I^{m-1}}^{u/u_i,v/v_j}$ is connected when $u_i=v_j$, for some $1\leq i,j\leq m$. 

For a monomial $u$ let $\supp(u)$ be the set $\{i\colon\ x_i\mid u\}$. First assume that $\supp(u_i)\cap\supp(v_j)$ is a singleton for all $1\leq i,j\leq m$. Without loss of generality, suppose that $u_1=x_1x_2x_3$ and $v_1=x_1x_4x_5$. Clearly, $U\subseteq\{x_1x_2x_3,\ x_2x_3x_4,\ x_2x_3x_5\}$ and $V\subseteq\{x_1x_4x_5,\ x_2x_4x_5,\ x_3x_4x_5\}$. Since $\ind((u_1, X))=1$ for any $\varnothing\neq X\subseteq V$  (see Theorem \ref{index of gap-free}), we conclude that $\G(I)\setminus(\{u_1\}\cup V)$ (and similarly $\G(I)\setminus(\{v_1\}\cup U)$) is non-empty. Let $x_ax_bx_c\in \G(I)\setminus(\{u_1\}\cup V)$. If $x_ax_bx_c\notin U$, then we may assume w.l.o.g. that $x_ax_bx_c=x_1x_2x_4$. Let $w_1=(u/x_3)x_4$ and $w_2=(v/x_5)x_2$. Then $w_1,w_2\in\G(I^m)$ and divide $\lcm(u,v)$. Since $w_1=x_1x_2x_4(u/u_1)$ and $w_2=x_1x_2x_4(v/v_1)$ share the same element $x_1x_2x_4\in\G(I)$ and the graph $G_{I^{m-1}}^{w_1/x_1x_2x_4,w_2/x_1x_2x_4}$ is connected, there exists a path between $w_1$ and $w_2$ in $\G_{I^m}^{u,v}$. Since $u\sim w_1$ and $w_2\sim v$, it follows that $u$ and $V$ are connected via a path in $G_{I^m}^{u,v}$, which is a contradiction. Thus $\G(I)\subseteq U\cup V$. Clearly, $|\G(I)\cap U|,|\G(I)\cap V|\geq2$. W.l.o.g. we can assume that $x_2x_3x_4\in \G(I)\cap U$ and $x_2x_4x_5\in\G(I)\cap V$. Let $w_1=(u/x_1)x_4$, $w_2=(u/x_1x_3)x_4x_5$, and $w_3=(v/x_1)x_2$. Clearly, $w_1,w_2,w_3\in\G(I^m)$ and divide $\lcm(u,v)$. Notice that $w_2=x_2x_4x_5(u/u_1)$ and $w_3=x_2x_4x_5(v/v_1)$ share the same element $x_2x_4x_5$ of $\G(I)$. Since $u\sim w_1\sim w_2$, $w_3\sim v$, and $w_2/x_2x_4x_5$ and $w_3/x_2x_4x_5$ are connected via a path in $G_{I^{m-1}}^{w_2/x_2x_4x_5, w_3/x_2x_4x_5}$, it follows that $u$ and $v$ are connected via a path in $\G_{I^m}^{u,v}$. This contradiction shows that there exist $i,j$ such that $u_i$ and $v_j$ have two variables in common, say $i=j=1$, $u_1=x_1x_2x_3$, and  $v_1=x_1x_2x_4$. Let $u=x_1^{\alpha_1}\cdots x_5^{\alpha_5}$ and $v=x_1^{\beta_1}\cdots x_5^{\beta_5}$. We claim
\begin{itemize}
\item[($*$)]If $u_s=x_ix_jx_k$ and $v_t=x_ix_jx_l$, then $\alpha_k\leq\beta_k$ and $\alpha_l\geq\beta_l$.
\item[]If not, we may assume $\alpha_k>\beta_k$. Then $v\sim (v/v_t)u_s$, $(v/v_t)u_s$ divides $\lcm(u,v)$, and $u/u_s$ and $v/v_t$ are connected via a path in $G_{I^{m-1}}^{u/u_s,v/v_t}$. Hence $u$ and $v$ are connected via a path in $G_{I^m}^{u,v}$, which contradicts our assumption.
\end{itemize}

Since $u_1=x_1x_2x_3$ and $v_1=x_1x_2x_4$ we get $\alpha_4\geq\beta _4>0$ and $0<\alpha_3\leq\beta_3$. This shows that $x_4\mid u$ and $x_3\mid v$. By symmetry, we have three cases to consider:

Case 1. $x_4x_5\nmid u_i$ and $x_3x_5\nmid v_i$, for $i=1,\ldots,m$. Since $x_4\mid u$ and $x_3\mid v$, we can assume that $u_2=x_1x_3x_4$ and $v_2=x_2x_3x_4$. As $\alpha_3\leq\beta_3<m$ and $\beta_4\leq\alpha_4<m$, there exist $3\leq i,j\leq m$ such that $x_3\nmid u_i$ and $x_4\nmid v_j$. From the inequalities $u_i\neq x_1x_2x_4$ and $v_j\neq x_1x_2x_3$, we observe that $x_5\mid u_i,v_j$. By assumption, $x_4\nmid u_i$ and $x_3\nmid v_j$. Thus $u_i=v_j=x_1x_2x_5$, which is a contradiction.

Case 2. $x_4x_5\mid u_i$ and $x_3x_5\mid v_j$ for some $1\leq i,j\leq m$. W.l.o.g. assume that $i=j=2$. By symmetry of $1$ and $2$ in $u_1$ and $v_1$, either $u_2=x_1x_4x_5$ or $u_2=x_3x_4x_5$. First assume that $u_2=x_1x_4x_5$. Then $\alpha_2\geq\beta_2$ and $\alpha_5\leq\beta_5$ by $(*)$. We have three cases for $v_2$. If $v_2=x_1x_3x_5$, then we get $\alpha_i=\beta_i$ for all $1\leq i\leq 5$ by $(*)$ and hence $u=v$, a contradiction. If $v_2=x_2x_3x_5$, then $\alpha_5=\beta_5$ by $(*)$. Since $\alpha_1\leq\beta_1<m$, $u_i\mid x_2x_3x_4x_5$ for some $i>2$. Analogously, $v_j\mid x_1x_3x_4x_5$ for some $j>2$ as $\beta_2\leq\alpha_2<m$. As $U\cap V=\varnothing$, we deduce that $x_4\mid u_i$ and $x_3\mid v_j$, from which it follows that $\alpha_4=\beta_4$ and $\alpha_3=\beta_3$ by $(*)$, respectively. If $x_2\mid u_i$ (resp. $x_1\mid v_1$), then $(*)$ yields $\alpha_1=\beta_1$ (resp. $\alpha_2=\beta_2$) and hence $\alpha_2=\beta_2$ (resp. $\alpha_1=\beta_1$) as $\deg(u)=\deg(v)$. Thus $u=v$, which is a contradiction. Hence $x_2\nmid u_i$ and $x_1\nmid v_j$, which implies that $u_i=v_j=x_3x_4x_5$ contradicting the fact that $U\cap V=\varnothing$. The case $v_3=x_3x_4x_5$ is ruled out with the same discussion as in the case $v_2=x_2x_3x_5$. Now, assume $u_2=x_3x_4x_5$. By symmetry, either $v_2=x_1x_3x_5$ or $v_2=x_3x_4x_5$. The case $v_2=x_1x_3x_5$ leads us to a contradiction as in the case $u_2=x_1x_4x_5$. Hence $v_2=x_3x_4x_5$, which is a contradiction as $U\cap V=\varnothing$.

Case 3. $x_4x_5\mid u_i$ for some $2\leq i\leq m$ but $x_3x_5\nmid v_j$, for all $2\leq j\leq m$. W.l.o.g. assume that $i=2$. By symmetry of $x_1$ and $x_2$ in $u_1$ and $v_1$, either $u_2=x_1x_4x_5$ or $u_2=x_3x_4x_5$. Since $\beta_3\geq\alpha_3>0$, $x_3\mid v_j$ for some $2\leq j\leq m$, say $x_3\mid v_2$. As $x_5\nmid v_2$ and $v_2\neq x_1x_2x_3$, it follows that $x_4\mid v_2$. Thus $v_2=x_1x_3x_4$ or $v_2=x_2x_3x_4$. First assume that $u_2=x_1x_4x_5$. Then $\alpha_2\geq\beta_2$ and $\alpha_5\leq\beta_5$ by $(*)$. From $\beta_5\geq\alpha_5>0$, it follows that $x_5\mid v_j$ for some $3\leq j\leq m$, say $x_5\mid v_3$. As $x_3\nmid v_3$ and $v_3\neq x_1x_4x_5$, we must have $x_2\mid v_3$. Hence $v_3=x_1x_2x_5$ or $v_3=x_2x_4x_5$. If $v_2=x_1x_3x_4$, then by swapping $u_1$ with $u_2$, and $v_2$ with $v_3$, we are at the position of Case 2, which leads us to a contradiction. Hence $v_2=x_2x_3x_4$ and consequently $\alpha_1\leq\beta_1$ by $(*)$. If $v_3=x_1x_2x_5$, then by replacing $v_1$ and $v_2$ by $v_2$ and $v_3$, respectively, we get a contradiction as we are at the position of Case 2 once again. Thus $v_3=x_2x_4x_5$. From the inequalities $\alpha_1\leq\beta_1<m$, we conclude that $x_1\nmid u_i$ for some $3\leq i\leq m$, say $x_1\nmid u_3$. Since $U\cap V=\varnothing$, it follows that $x_3x_5\mid u_3$. Now by replacing $u_1,u_2,v_1,v_2$ by $v_1,v_3,u_1,u_3$, we are at the position of Case 2 leading us to a contradiction. This shows that $u_2\neq x_1x_4x_5$, that is we must have $u_2=x_3x_4x_5$. Then by swapping $u_1$ with $u_2$, and $v_1$ with $v_2$, we arrive at the case where $u_2=x_1x_4x_5$ after suitable relabeling of the variables, which ruled out before. The proof is complete.
\end{proof}
\begin{remark}
Theorem \ref{index with n <= 5} is not valid for $n>5$, or monomial ideals with $n\leq 5$ that are not square-free. In Theorem \ref{index(I)>1 and index(I2)>1} we present (all possible) square-free monomial ideals $I\subseteq\KK[x_1,\ldots,x_6]$ with $\ind(I)>1$ and $\ind(I^2)=1$. As an instance of a monomial ideal $I$ satisfying $\ind(I)>1$ and $\ind(I^2)=1$ in $n\leq5$ variables, we can refer to Conca's example $I=(x_1^2x_2, x_1^2x_3, x_1x_3^2, x_2x_3^2, x_1x_3x_4)\subseteq\KK[x_1,x_2,x_3,x_4]$ (see \cite[Example 2.5]{ac}).
\end{remark}
\section{Index of the square of linearly presented cubic square-free monomial ideals}
In the rest of the paper, we characterize all cubic square-free monomial ideals $I$ such that $\ind(I)>1$ and $\ind(I^2)=1$. To this end, we need some preparations we mention here.
\begin{lemma}
Let $I$ be a monomial ideal, $k\geq1$, and $u=u_1\cdots u_k$ and $v=v_1\cdots v_k$ be monomials in $\G(I^k)$, where $u_1,\ldots,u_k,v_1,\ldots,v_k\in \mathcal{G}(I)$. Let $A,B\subseteq[k]$ be non-empty sets such that $|A|+|B|=k$ and $\prod_{i\in A}u_i\prod_{j\in B}v_j$ divides $\lcm(u,v)$. If $\ind(I^{|A|})>1$ and $\ind(I^{|B|})>1$, then $u$ and $v$ are connected via a path in $G_{I^k}^{u,v}$.
\end{lemma}
\begin{proof}
Without loss of generality, we may assume that $A=\{1,\ldots,r\}$ and $B=\{r+1,\ldots,k\}$. Let $u':=u_1\cdots u_r$, $u'':=u_{r+1}\cdots u_k$, $v':=v_1\cdots v_r$, and $v'':=v_{r+1}\cdots v_k$. Suppose $u'=\alpha_0,\alpha_1,\ldots,\alpha_t=v'$ is a path in $G_{I^{|A|}}^{u',v'}$ and $u''=\beta_0,\beta_1,\ldots,\beta_s=v''$ is a path in $G_{I^{|B|}}^{u'',v''}$. Then 
\[u=u'\beta_0,u'\beta_1,\ldots,u'\beta_s=\alpha_0v'',\alpha_1v'',\ldots,\alpha_tv''=v\]
is a path in $G_{I^k}^{u,v}$, as required.
\end{proof}
\begin{corollary}\label{prod(u_i)prod(v_j)|lcm(u,v)}
Let $I$ be a monomial ideal and $k\geq1$ be such that $\ind(I^i)>1$ for $i=1,\ldots,k-1$. Let $u=u_1\cdots u_k$ and $v=v_1\cdots v_k$  be monomials in $\G(I^k)$, where $u_1,\ldots,u_k,v_1,\ldots,v_k\in \mathcal{G}(I)$. If $u,v$ belong to distinct connected components of $G_{I^k}^{u,v}$, then $\prod_{i\in A}u_i\prod_{j\in B}v_j$ does not divide $\lcm(u,v)$ for any two non-empty subsets $A$ and $B$ of $[k]$ with $|A|+|B|=k$.
\end{corollary}

Let $\D_i^n$ be the $3$-clutter associated to the $i$-th vertex in Figure \ref{d=3, n=6}, \ref{d=3, n=7}, or \ref{d=3, n=8} when $n=6$, $n=7$, or $n=8$, respectively, where
\begin{align*}
\D_1^6&=\{123,\ 246,\ 145,\ 356,\ 134,\ 136,\ 146,\ 346\},\\
\D_6^7&=\{123,\ 124,\ 127,\ 145,\ 147,\ 247,\ 267,\ 347\},\\
\D_{48}^7&=\{123,\ 124,\ 136,\ 145,\ 146,\ 147,\ 167,\ 246,\ 267,\ 346,\ 347\},\\
\D_1^8&=\{123,\ 124,\ 125,\ 145,\ 147,\ 246,\ 248,\ 258,\ 456\},
\end{align*}
and an arrow with label $abc$ from $i$-th vertex to $j$-th vertex means $\D_j^n=\D_i^n\cup\{abc\}$. Notice that all parallel edges in Figures \ref{d=3, n=6}, \ref{d=3, n=8}, \ref{d=3, n=7} have the same labels and that all clutters $\D_i^n$ are uniquely determined by diagram chasing. Let $\mathscr{D}$ denote the set of all $\D_i^n$ for $n=6,7,8$. The ideal associated to the clutter $\D_1^6$, introduced by Sturmfels \cite{bs} in 2000, is the first example of an ideal with linear resolution (quotients) whose square does not have linear resolution.

\begin{figure}[H]
\begin{minipage}{0.4\textwidth}
\centering
\begin{tikzpicture}
\node [very thick, draw, inner sep=2pt] (1) at (0,0) {{\tiny $1$}};
\node [draw, inner sep=2pt] (2) at (1,0) {{\tiny $2$}};
\node [draw, inner sep=2pt] (3) at (2,0) {{\tiny $3$}};
\node [draw, inner sep=2pt] (4) at (3,0) {{\tiny $4$}};
\node [draw, inner sep=2pt] (5) at (4,0) {{\tiny $5$}};
\node [draw, inner sep=2pt] (6) at (1,1) {{\tiny $6$}};
\draw[->-] (1) to node [below] {{\tiny $126$}} (2);
\draw[->-] (2) to node [below] {{\tiny $156$}} (3);
\draw[-<-] (3) to node [below] {{\tiny $146$}} (4);
\draw[-<-] (4) to node [below] {{\tiny $136$}} (5);
\draw[->-] (2) to node [above, rotate=90] {{\tiny $234$}} (6);
\end{tikzpicture}
\caption{$n=6$}\label{d=3, n=6}
\end{minipage}
\begin{minipage}{0.4\textwidth}
\centering
\begin{tikzpicture}
\node [very thick, draw, inner sep=2pt] (1) at (0,0) {{\tiny $1$}};
\node [draw, inner sep=2pt] (2) at (1,0) {{\tiny $2$}};
\node [draw, inner sep=2pt] (3) at (2,0) {{\tiny $3$}};
\node [draw, inner sep=2pt] (4) at (0,1) {{\tiny $4$}};
\node [draw, inner sep=2pt] (5) at (1,1) {{\tiny $5$}};
\node [draw, inner sep=2pt] (6) at (2,1) {{\tiny $6$}};
\draw[->-] (1) to node [below] {{\tiny $126$}} (2);
\draw[->-] (2) to node [below] {{\tiny $167$}} (3);
\draw[->-] (4)--(5);
\draw[->-] (5)--(6);
\draw[->-] (1) to node [above, rotate=90] {{\tiny $148$}} (4);
\draw[->-] (2)--(5);
\draw[->-] (3)--(6);
\end{tikzpicture}
\caption{$n=8$}\label{d=3, n=8}
\end{minipage}
\end{figure}

\begin{figure}[H]
\begin{tikzpicture}[scale=1.5]
\node [draw, inner sep=2pt] (1) at (0,0) {{\tiny $1$}};
\node [draw, inner sep=2pt] (2) at (1,0) {{\tiny $2$}};
\node [draw, inner sep=2pt] (3) at (2,0) {{\tiny $3$}};
\node [draw, inner sep=2pt] (4) at (4.75,0) {{\tiny $4$}};
\node [draw, inner sep=2pt] (5) at (6,0) {{\tiny $5$}};
\node [very thick, draw, inner sep=2pt] (6) at (7,0) {{\tiny $6$}};
\node [draw, inner sep=2pt] (7) at (8,0) {{\tiny $7$}};
\node [draw, inner sep=2pt] (8) at (9,0) {{\tiny $8$}};

\node [draw, inner sep=2pt] (9) at (0,1) {{\tiny $9$}};
\node [draw, inner sep=2pt] (10) at (1,1) {{\tiny $10$}};
\node [draw, inner sep=2pt] (11) at (2,1) {{\tiny $11$}};
\node [draw, inner sep=2pt] (12) at (4.75,1) {{\tiny $12$}};
\node [draw, inner sep=2pt] (13) at (6,1) {{\tiny $13$}};
\node [draw, inner sep=2pt] (14) at (7,1) {{\tiny $14$}};
\node [draw, inner sep=2pt] (15) at (8,1) {{\tiny $15$}};
\node [draw, inner sep=2pt] (16) at (9,1) {{\tiny $16$}};

\node [draw, inner sep=2pt] (17) at (0,2) {{\tiny $17$}};
\node [draw, inner sep=2pt] (18) at (1,2) {{\tiny $18$}};
\node [draw, inner sep=2pt] (19) at (2,2) {{\tiny $19$}};
\node [draw, inner sep=2pt] (20) at (4.75,2) {{\tiny $20$}};
\node [draw, inner sep=2pt] (21) at (6,2) {{\tiny $21$}};
\node [draw, inner sep=2pt] (22) at (7,2) {{\tiny $22$}};
\node [draw, inner sep=2pt] (23) at (8,2) {{\tiny $23$}};
\node [draw, inner sep=2pt] (24) at (9,2) {{\tiny $24$}};

\node [draw, inner sep=2pt] (25) at (2.75,-0.75) {{\tiny $25$}};
\node [draw, inner sep=2pt] (26) at (2.75,0.25) {{\tiny $26$}};
\node [draw, inner sep=2pt] (27) at (2.75,1.25) {{\tiny $27$}};

\node [draw, inner sep=2pt] (28) at (4,0.75) {{\tiny $28$}};
\node [draw, inner sep=2pt] (29) at (4,1.75) {{\tiny $29$}};
\node [draw, inner sep=2pt] (30) at (4,2.75) {{\tiny $30$}};

\node [draw, inner sep=2pt] (31) at (3.25,1.5) {{\tiny $31$}};
\node [draw, inner sep=2pt] (32) at (3.25,2.5) {{\tiny $32$}};
\node [draw, inner sep=2pt] (33) at (3.25,3.5) {{\tiny $33$}};

\node [draw, inner sep=2pt] (34) at (5.25,0.75) {{\tiny $34$}};
\node [draw, inner sep=2pt] (35) at (5.25,1.75) {{\tiny $35$}};
\node [draw, inner sep=2pt] (36) at (5.25,2.75) {{\tiny $36$}};

\node [draw, inner sep=2pt] (37) at (6,3) {{\tiny $37$}};
\node [draw, inner sep=2pt] (38) at (6,4) {{\tiny $38$}};
\node [draw, inner sep=2pt] (39) at (5.25,3.75) {{\tiny $39$}};
\node [draw, inner sep=2pt] (40) at (5.25,4.75) {{\tiny $40$}};

\node [draw, inner sep=2pt] (41) at (6,-1) {{\tiny $41$}};
\node [draw, inner sep=2pt] (42) at (6,5) {{\tiny $42$}};

\node [draw, inner sep=2pt] (43) at (1.75,-0.75) {{\tiny $43$}};
\node [draw, inner sep=2pt] (44) at (2.50,-1.50) {{\tiny $44$}};
\node [draw, inner sep=2pt] (45) at (3.25,-2.25) {{\tiny $45$}};
\node [draw, inner sep=2pt] (46) at (3.50,-1.50) {{\tiny $46$}};

\node [draw, inner sep=2pt] (47) at (1,3) {{\tiny $47$}};

\node [very thick, draw, inner sep=2pt] (48) at (8,4) {{\tiny $48$}};

\draw[->-] (1) to node [below] {{\tiny $124$}} (2);
\draw[->-] (2) to node [below] {{\tiny $247$}} (3);
\draw[-<-] (3) to node [below] {{\tiny $234$}} (4);
\draw[-<-] (4) to node [below] {{\tiny $245$}} (5);
\draw[-<-] (5) to node [below] {{\tiny $246$}} (6);
\draw[->-] (6) to node [below] {{\tiny $234$}} (7);
\draw[->-] (7) to node [below] {{\tiny $137$}} (8);

\draw[->-] (9)--(10);
\draw[->-] (10)--(11);
\draw[-<-] (11)--(12);
\draw[-<-] (12)--(13);
\draw[-<-] (13)--(14);
\draw[->-] (14)--(15);
\draw[->-] (15)--(16);

\draw[->-] (17)--(18);
\draw[->-] (18)--(19);
\draw[-<-] (19)--(20);
\draw[-<-] (20)--(21);
\draw[-<-] (21)--(22);
\draw[->-] (22)--(23);
\draw[->-] (23)--(24);

\draw[->-] (1) to node [above, rotate=90] {{\tiny $257$}} (9);
\draw[->-] (9) to node [above, rotate=90] {{\tiny $146$}} (17);
\draw[->-] (2)--(10);
\draw[->-] (10)--(18);
\draw[->-] (3)--(11);
\draw[->-] (11)--(19);
\draw[->-] (4)--(12);
\draw[->-] (12)--(20);
\draw[->-] (5)--(13);
\draw[->-] (13)--(21);
\draw[->-] (6)--(14);
\draw[->-] (14)--(22);
\draw[->-] (7)--(15);
\draw[->-] (15)--(23);
\draw[->-] (8)--(16);
\draw[->-] (16)--(24);

\draw[->-] (25)--(26);
\draw[->-] (26)--(27);

\draw[-<-] (3)--(25);
\draw[-<-] (11)--(26);
\draw[-<-] (19)--(27);

\draw[->-] (28)--(29);
\draw[->-] (29)--(30);

\draw[->-] (31)--(32);
\draw[->-] (32)--(33);

\draw[->-] (34)--(35);
\draw[->-] (35)--(36);

\draw[->-] (4)--(28);
\draw[->-] (12)--(29);
\draw[->-] (20) to node [above, rotate=-50] {{\tiny $167$}} (30);

\draw[->-] (28)--(31);
\draw[->-] (29)--(32);
\draw[->-] (30) to node [above, rotate=-50] {{\tiny $157$}} (33);

\draw[->-] (5)--(34);
\draw[->-] (13)--(35);
\draw[->-] (21)--(36);

\draw[-<-] (28)--(34);
\draw[-<-] (29)--(35);
\draw[-<-] (30)--(36);

\draw[-<-] (21)--(37);
\draw[->-] (37)--(38);
\draw[-<-] (36) to node [above, rotate=90] {{\tiny $257$}} (39);
\draw[->-] (39) to node [above, rotate=90] {{\tiny $157$}} (40);

\draw[->-] (37)--(39);
\draw[->-] (38)--(40);

\draw[-<-] (2) to node [below, rotate=-50] {{\tiny $245$}} (43);

\draw[->-] (43) to node [below, rotate=-50] {{\tiny $146$}} (44);
\draw[->-] (25)--(46);
\draw[->-] (43)--(25);
\draw[->-] (44)--(46);

\draw[->-] (44) to node [below, rotate=-50] {{\tiny $136$}} (45);

\draw[-<-] (18)--(47);

\draw[->-] (5) to node [below, rotate=90] {{\tiny $157$}} (41);

\draw[->-] (38) to node [below, rotate=90] {{\tiny $257$}} (42);
\end{tikzpicture}
\caption{$n=7$}\label{d=3, n=7}
\end{figure}

Let $I$ be a monomial ideal in a polynomial ring generated in degree $d$. It follows from \cite[Theorem 2.1]{yazdan} that $\ind(I)>1$ if and only if $\beta_{1,d+2}(I)=\cdots=\beta_{1,2d}(I)=0$. In particular, if $I$ is generated in degree $3$, then $\ind(I^2)>1$ if and only if $\beta_{1,8}(I^2)=\cdots=\beta_{1,12}(I^2)=0$. The following theorem shows surprisingly that if $I=I(\C)$ is the edge ideal of a $3$-uniform clutter $\C$ with $\ind(I)>1$, then $\ind(I^2)>1$ if and only if $\beta_{1,8}(I^2)=0$.
\begin{theorem}\label{index(I)>1 and index(I2)>1}
Let $\C$ be a $3$-uniform clutter and let $I=I(\C)$ be such that $\ind(I)>1$. The following conditions are equivalent:
\begin{itemize}
\item[{\rm (a)}]$\C$ is $\mathscr{D}$-free;
\item[{\rm (b)}]$\ind(I^2)>1$;
\item[{\rm (c)}]$\beta_{1,8}(I^2)=0$.
\end{itemize}
\end{theorem}
\begin{proof}
Let $\C'$ be a $3$-uniform clutter on $[n]$ and $J=I(\C')$ be such that $\ind(J)>1$. Let $u,v\in\G(J^2)$ be monomials such that $\supp(uv)=[n]$. We show that $\C'$ is isomorphic to an element of $\mathscr{D}$ if and only if $u,v$ belong to distinct connected components of $G:=G_{J^2}^{u,v}$. Suppose $u$ and $v$ belong to distinct connected components of $G$. Let $u=u_1u_2$ and $v=v_1v_2$ with $u_1,u_2,v_1,v_2\in\mathcal{G}(J)$. Corollary~\ref{prod(u_i)prod(v_j)|lcm(u,v)} implies that $u_iv_j\nmid \lcm(u,v)$ for $i,j\in\{1,2\}$. As a result, $\gcd(u_i,v_j)\neq1$ for $i,j\in\{1,2\}$ from which it follows that $n\leq9$. Notice that $\{u_1,u_2\}\cap\{v_1,v_2\}=\varnothing$ otherwise $u_i=v_j$ for some $i,j$, say $u_1=v_1$ and consequently $u_1u_2=u_2v_1\mid\lcm(u,v)$, a contradiction. If $n=9$, then $u_1=x_1x_2x_3$, $u_2=x_1x_4x_5$, $v_1=x_1x_6x_7$, and $v_2=x_1x_8x_9$ after a relabeling of the variables. Then $u_1v_1\mid\lcm(u,v)$, which is a contradiction. On the other hand, Theorem~\ref{index with n <= 5} in conjunction with Theorem~\ref{index in terms of graph} reveal that $n\geq6$. Thus $6\leq n\leq 8$. 

By invoking Corollary \ref{prod(u_i)prod(v_j)|lcm(u,v)} and a computer search presented by Algorithm \ref{algorithm1}, we arrive at the following cases after a suitable relabeling of variables:
\begin{align*}
u_1=x_1x_2x_3,\ u_2=x_2x_4x_6,\ v_1=x_1x_4x_5,\ v_2=x_3x_5x_6,&&\text{if }n=6,\\
u_1=x_1x_2x_3,\ u_2=x_3x_4x_7,\ v_1=x_1x_4x_5,\ v_2=x_2x_6x_7,&&\text{if }n=7,\\
u_1=x_1x_2x_3,\ u_2=x_4x_5x_6,\ v_1=x_1x_4x_7,\ v_2=x_2x_5x_8,&&\text{if }n=8.
\end{align*}
For a given fixed $n\in\{6,7,8\}$, let $\C_1^n:=\{F\colon\ \bx_F\in\{u_1,u_2,v_1,v_2\}\}$, and define $\C_2^n$, $\C_3^n$, and $\C_4^n$ as follows:
\begin{itemize}
\item $\C_2^n$ is the set of all triples $F$ such that for $w_1:=\bx_F$ there exists $w_2\in\{u_1,u_2\}$ satisfying $w_1w_2\in V(G)$, $u\sim w_1w_2$, and $v_iw_j\mid\lcm(v,w_1w_2)$ as well as all triples $F$ such that for $w_1:=\bx_F$ there exists $w_2\in\{v_1,v_2\}$ satisfying $w_1w_2\in V(G)$, $v\sim w_1w_2$, and $u_iw_j\mid\lcm(u,w_1w_2)$.
\item $\C_3^n=A^n\cup B^n\cup C^n$, where $A^n$, $B^n$, and $C^n$ are as follows: $A^n$ is the set of all $2$-sets $\{F_1,F_2\}$ such that $\bx_{F_1}\bx_{F_2}\in V(G)$ and $u\sim\bx_{F_1}\bx_{F_2}\sim v$; $B^n$ is the set of all $2$-sets $\{F_1,F_2\}$ such that $\bx_{F_1}w_1,\bx_{F_2}w_2\in V(G)$ and $u\sim\bx_{F_1}w_1\sim\bx_{F_2}w_2\sim v$ for some $w_1,w_2\in\{u_1,u_2,v_1,v_2\}$; and $C^n$ is the set of all $2$-sets $\{F_1,F_2\}$ such that either $u=\bx_{F_1}\bx_{F_2}$ and $\bx_{F_i}v_j\mid\lcm(u,v)$, or $v=\bx_{F_1}\bx_{F_2}$ and $\bx_{F_i}u_j\mid\lcm(u,v)$, for some $i,j$.
\item $\C_4^n:=\binom{[n]}{3}\setminus(\C_1^n\cup\C_2^n\cup\bigcup\C_3^n)$.
\end{itemize}
Observe that $\C_2^n\cap\C=\varnothing$ and $\{F_1,F_2\}\nsubseteq\C$ for all $\{F_1,F_2\}\in\C_3^n$. Let $G_n$ be the graph obtained from $\C_3^n$ as a set of edges. Then $\C\cap V(G_n)$ is an independent set in $G_n$. It follows that, the possible candidates for $\C'$ are those of the form $\C_1^n\cup X\cup Y$, where $X$ is an independent set of $G_n$ and $Y$ is a subset of $\C_4^n$. In view of Algorithm \ref{algorithm2}, analysing all these cases yields $6$, $48$, and $6$ possible $3$-clutters $\C'$, namely $\D_i^n$'s, for which $\ind(J)>1$ and $u,v$ belong to different connected components of $G$ for $n=6$, $7$, and $8$, respectively. Hence $\C'$ is isomorphic to one of the elements of $\mathscr{D}$. Since the vertices $u,v$ belong to distinct connected components of $G_{I(\D_i^n)^2}^{u,v}$ by the argument above, the converse also holds.

Now let $I=I(\C)$ be the edge ideal of a $3$-uniform clutter $\C$ such that $\ind(I)>1$. Also, let $I_{u,v}$ be the edge ideal of $\C[\supp(uv)]$ for all $u,v\in\G(I^2)$. Theorem \ref{index in terms of graph} states that $\ind(I^2)>1$ if and only if $\ind(I_{u,v}^2)>1$ for all monomials $u,v\in\G(I^2)$ or equivalently $u,v$ belong to the same connected component of $G_{I_{u,v}^2}^{u,v}$ for all monomials $u,v\in\G(I^2)$. Our discussions above show that (a) and (b) are equivalent. Clearly, (b) implies (c). It remains to prove that (c) implies (a). Note that 
\begin{align*}
\beta_{1,8}(I^2)=\sum_{\substack{w\in L(I^2)\\\deg w=8}}\beta_{1,w}(I^2)&=\sum_{\substack{w\in L(I^2)\\\deg w=8}}\dim_\KK(\tilde{H}_0(\Delta_{I^2}(1,w);\KK))\\
&\geq\dim_\KK(\tilde{H}_0(\Delta_{I^2}(1,\lcm(u,v));\KK))\\
&=\dim_\KK(\tilde{H}_0(\Delta_{I_{u,v}^2}(1,\lcm(u,v));\KK))
\end{align*}
for all $u,v\in\G(I^2)$. If $\C$ contains an induced subclutter isomorphic to some $\D_i^n$ of $\mathscr{D}$, then there exist $u,v\in\G(I^2)$ such that $\D_i^n\cong\C[\supp(uv)]$ and $u,v$ belong to distinct connected components of $G_{I_{u,v}^2}^{u,v}=G_{I^2}^{u,v}$. Hence, by Theorem \ref{index in terms of graph}, $\Delta_{I^2_{u',v'}}(1,\lcm(u',v'))=\Delta_{I^2}(1,\lcm(u',v'))$ is disconnected for some $u',v'\in\G(I^2)$, which implies that $\beta_{1,8}(I^2)\neq0$. The proof is complete.
\end{proof}

Let $G_F^n$ be the stabilizer of $\{\{1,2,3\}, F\}$ in $S_n$ on its action on $2^{\binom{[n]}{3}}$ for $n\in\{6,7,8\}$, and $F=\{1,2,4\}$ or $F=\{1,4,5\}$. It is easy to see that 
$G_{\{1,2,4\}}^n=\gen{(1\ 2), S_{\{4,\ldots,n\}}}$ and $G_{\{1,4,5\}}^n=\gen{(2\ 3),(4\ 5),S_{\{6,\ldots,n\}}}$. The following algorithms are used in the proof of Theorem \ref{index(I)>1 and index(I2)>1}. The source codes of algorithms in GAP are available in \cite{mfdg}.
\begin{algorithm}[H]
\begin{algorithmic}[1]\baselineskip=10pt\relax
\REQUIRE Positive integer $n\in\{6,7,8\}$
\ENSURE All quadruples of cubic square-free monomials $(u_1,u_2,v_1,v_2)$ such that $\deg\lcm(u_1u_2,v_1v_2)>8$ and $u_iv_j\nmid\lcm(u_1u_2,v_1v_2)$ for all $i,j\in\{1,2\}$ modulo relabeling of indices
\STATE $\S \longleftarrow \binom{[n]}{3}$
\STATE $U_1 \longleftarrow \{1,2,3\}$
\FOR{$U_1\in\{\{1,2,4\},\ \{1,4,5\}\}$}
\STATE $\T_{V_1} \longleftarrow \varnothing$
\FOR{$U_2\in\S$}
\FOR{$V_2\in\S$}
\STATE $u \longleftarrow \bx_{U_1}\bx_{U_2}$ and $v \longleftarrow \bx_{V_1}\bx_{V_2}$
\IF{$\bx_{U_i}\bx_{V_j}\nmid\lcm(u,v)$ for all $i,j\in\{1,2\}$}
\STATE $\T_{V_1} \longleftarrow \T_{V_1}\cup\{(\{U_1,U_2\},\{V_1,V_2\})\}$
\ENDIF
\ENDFOR
\ENDFOR
\STATE $\O_{V_1} \longleftarrow$ the set of representatives of orbits of $G_{V_1}^n$ on $\T_{V_1}$
\STATE $\O_{V_1} \longleftarrow \O_{V_1}\setminus\{(\{U_1,U_2\},\{V_1,V_2\})\colon |(U_1\cup U_2)\cap(V_1\cup V_2)|=5\}$
\ENDFOR
\RETURN $\{(\bx_{U_1},\bx_{U_2},\bx_{V_1},\bx_{V_2})\colon (\{U_1,U_2\},\{V_1,V_2\})\in\O_{\{1,2,4\}}\cup\O_{\{1,4,5\}}\}$
\end{algorithmic}
\caption{}
\label{algorithm1}
\end{algorithm}
\begin{algorithm}[H]
\begin{algorithmic}[1]\baselineskip=10pt\relax 
\REQUIRE Any $(u_1,u_2,v_1,v_2)$ returned by Algorithm \ref{algorithm1} for a positive integer $n\in\{6,7,8\}$
\ENSURE All $3$-uniform clutters $\D$ on $[n]$ such that $J=I(\D)$ contains $u_1,u_2,v_1,v_2$ and satisfies $\ind(J)>1$ and $\ind(J^2)=1$
\STATE $u \longleftarrow u_1u_2$ and $v \longleftarrow v_1v_2$
\STATE Compute $\C_2^n$, $\C_3^n$, $\C_4^n$
\STATE $G_n \longleftarrow$ simple graph on $\binom{[n]}{3}$ with edge set $\C_3^n$
\STATE $i \longleftarrow 0$
\FOR{$X$ an independent set of $G_n$}
\FOR{$Y\subseteq \C_4^n$}
\STATE $\C' \longleftarrow \C_1^n\cup X\cup Y$
\STATE $J \longleftarrow I(\C')$
\IF{$\bar{\C'}$ is $\mathscr{C}$-free}
\IF{$u,v$ are in distinct connected components of $G_{J^2}^{u,v}$}
\STATE $i \longleftarrow i + 1$
\STATE $\D_i^n \longleftarrow \C'$
\ENDIF
\ENDIF
\ENDFOR
\ENDFOR
\RETURN $\D_1^n,\ldots,\D_i^n$
\end{algorithmic}
\caption{}
\label{algorithm2}
\end{algorithm}

Let $G$ be a graph and $I=I(G)$ be the edge ideal of $G$. If $\bar{G}$ is $C_4$-free then $\mathrm{index}(I^k)>1$ for all $k$ (Theorem~\ref{main theorem of MJR}). Francisco, H\`{a}, and Van Tuyl raised the question whether in this case $\mathrm{index}(I^k)= \infty$ for all $k \geq 2$? (see \cite[Question 1.7]{en-ip}). Nevo and Peeva show by an example that this statement is not true in general (see \cite[Couterexample 1.10]{en-ip}). Nevo \cite{en} shows that $\mathrm{index}(I^2)= \infty$ if $G$ is both gap-free and claw-free. Also, Banerjee \cite[Theorem 1.2]{ab} shows that $\mathrm{index}(I^k)= \infty$ for all $k$ if $G$ is both gap-free and cricket-free. In 	this regard, we pose the following question in the case of $3$-uniform clutters instead of graphs.
\begin{question}
Let $\C$ be a $3$-uniform clutter and $I=I(\C)$ be the edge ideal of $\C$. Assume that $\bar{\C}$ is $\mathscr{C}$-free and $\C$ is $\mathscr{D}$-free.
\begin{itemize}
\item[\rm (i)] Under which conditions $I^2$ has a linear resolution?
\item[\rm (ii)] Under which conditions $I^k$ has a linear resolution for all $k$?
\end{itemize} 
\end{question}

For non-negative integers $d\geq1$ and $k\geq0$, let $\Omega_{d,k}$ denote the set of all $d$-uniform clutters $\C$ satisfying the following conditions:
\begin{itemize}
\item[(i)]$\ind(I(\C)),\ldots,\ind(I(\C)^{k-1})>1$,
\item[(ii)]$\ind(I(\C)^k)=1$,
\item[(iii)]no proper induced sub-clutter of $\C$ satisfies (i) and (ii).
\end{itemize}
Also, let $\Omega_{d,k}(n)$ be the number of isomorphism classes of $d$-uniform clutters of $\Omega_{d,k}$ with $n$ vertices, for all $n\geq 1$. It turns out that, 
\begin{align*}
\{\Omega_{2,1}(n)\}&=0,0,0,1,0,0,\ldots,\\
\{\Omega_{3,1}(n)\}&=0,0,0,0,3,1,0,0,\ldots,\\
\{\Omega_{3,2}(n)\}&=0,0,0,0,0,6,48,6,0,0,\ldots
\end{align*}

\begin{conjecture}
With the notation as above, the sequence $\{\Omega_{d,k}(n)\}_n$ is unimodal.
\end{conjecture}

\end{document}